\documentclass{amsart}

\newcommand{\vertiii}[1]{{\left\vert\kern-0.25ex\left\vert\kern-0.25ex\left\vert #1 
    \right\vert\kern-0.25ex\right\vert\kern-0.25ex\right\vert}}

\usepackage{url}
\usepackage[initials]{amsrefs}
\usepackage{amsmath,amsthm,amscd}
\usepackage{enumerate}

\usepackage{amssymb}
\usepackage{mathrsfs}
\usepackage{graphicx}

\usepackage{color}
\usepackage{mathtools}

\theoremstyle{plain}
\newtheorem{maintheorem}{Theorem}
\newtheorem{thm}{Theorem}[section]

\newtheorem{lem}[thm]{Lemma}
\newtheorem{prop}[thm]{Proposition}

\theoremstyle{definition}

\theoremstyle{remark}
\newtheorem{remark}[thm]{Remark}

\numberwithin{equation}{section}

\title[Irregular sets for piecewise monotonic maps]{
Irregular sets for piecewise monotonic maps
 }

\date{\today}

\author{Yushi Nakano}
\address[Yushi Nakano]{Department of Mathematics, Tokai University, 
Kanagawa 259-1292, JAPAN}
\email{yushi.nakano@tsc.u-tokai.ac.jp}

\author{Kenichiro Yamamoto}
\address[Kenichiro Yamamoto]{Department of General Education, Nagaoka University of Technology, Niigata 940-2188, JAPAN}
\email{k\_yamamoto@vos.nagaokaut.ac.jp}

\subjclass[2010]{37B10,37B40,37E05}

\keywords{Irregular set, historic behavior, piecewise monotonic map,
coding space, Markov Diagram}

\begin{document}

\begin{abstract}
For any transitive piecewise monotonic map
for which  the set of periodic measures is dense in the set of ergodic invariant measures  (such as
monotonic mod one transformations and piecewise monotonic maps with two monotonic pieces), 
 we show that the set of points for which the Birkhoff average of a  continuous function 
  does not exist (called the irregular set)
   is either empty or has full topological entropy.
This 
generalizes Thompson's theorem for irregular sets of $\beta$-transformations, and 
reduces a complete description 
 of irregular sets of transitive piecewise monotonic maps to 
   Hofbauer-Raith problem on the density of periodic measures.
\end{abstract}

\maketitle

\section{Introduction}\label{section:introduction}

Let $X$ be a compact metric space  and $T: X\to X$ a Borel measurable map.
Let $\mathcal C(X)$ be the set of 
 continuous functions on $X$.
The \emph{irregular set} $E(\varphi )$ of $\varphi \in \mathcal C(X)$ is given by
\[
E(\varphi ) = \left\{ x\in X \mid \lim _{n\to \infty } \frac{1}{n} \sum _{j=0}^{n-1} \varphi (T^j (x)) \; \text{does not exist} \right\} .
\]
It is also called the set of \emph{non-typical points}  
 (\cite{BS2000}) 
  or \emph{divergence points} 
   (\cite{CKS2005}),
   and (the forward orbit of) a point in $\bigcup _{\varphi \in \mathcal C (X)}E(\varphi ) $ 
    is said to have \emph{historic behavior} (\cite{Ruelle2001, Takens2008}). 

Although every irregular set is a $\mu$-zero measure set for any invariant measure $\mu$ due to Birkhoff's ergodic theorem, the set is known to be remarkably large for abundant dynamical systems.
Pesin and Pitskel \cite{PP1984}   obtained the first result for the largeness of irregular sets from thermodynamic  viewpoint.  
In their paper, they showed that    every irregular set for the full shift is 
 either empty or has full topological entropy, that is, 
\[
 E(\varphi ) =\emptyset \quad \text{or} \quad h_{\mathrm{top}}(T,E(\varphi ) ) = h_{\mathrm{top}}(T,X)
\]
(they also showed that $E(\varphi)$ has  full Hausdorff dimension  if and only if $E(\varphi ) \neq \emptyset$).
Here $h_{\mathrm{top}}(T,Z)$ is the (Bowen's Hausdorff) topological entropy for a continuous map $T$ on a (not necessarily compact) Borel set $Z$ given in \cite{Bowen1973} (see Subsection \ref{subsection:defofentropy} for precise definition; 
refer to  \cite{HNP2008} for relation between entropies for a non-compact). 
It is also known that $E(\varphi ) \neq \emptyset$ if and only if $ \int \varphi d\mu _1 \neq \int \varphi d\mu _2$ for some ergodic invariant probability measures $\mu _1, \mu _2$ (refer to \cite[Lemma 1.6]{Thompson2010}).
Pesin-Pitskel's
 thermodynamic dichotomy for irregular sets was extended to topologically mixing subshifts of finite type in \cite{BS2000} (together with the detailed study of the set of points at which Lyapunov exponent or local entropy fail to exist), 
to continuous maps with specification property in \cite{CKS2005} (see also \cite{Thompson2010}), 
and  to continuous maps with almost specification property (including all $\beta$-transformations) by Thompson \cite{Thompson2012}.
We also note that F\"arm \cite{Farm2011} showed a stronger property (called large intersection property) than the full Hausdorff dimension for irregular sets of $\beta$-transformations for a large class of $\beta$'s (this restriction on $\beta$ was later removed by in \cite{FP2013}), independently of the Thompson's work. 
See also \cite{Ruelle2001, Takens2008, KS2017, AP2019, BLV2014, FP2011,  BV2015, BV2017} and references therein for the study of irregular sets from other viewpoints.

The aim of this paper is to extend the 
 thermodynamic  dichotomy 
 to transitive piecewise monotonic maps 
  for which the set of  periodic measures is dense in the set of ergodic invariant measures. 
We emphasize that we do not (explicitly) assume any specification-like property on $T$, being in contrast to all the previous works
(see also the remark above Proposition \ref{prop:1}).
Furthermore, we will see that the transitivity seems to be intractable  (Remark \ref{rmk:1021};  see also Remark \ref{rmk:3.2} for necessity of  
the density of
 periodic measures)
 and that the density of periodic measures 
is shown to hold for abundant classes of transitive piecewise monotonic maps 
   by several authors  while no transitive piecewise monotonic map without the density of periodic measures is presently known at now 
  (see Subsection \ref{subsection:application} for detail).
Hence we hope that our result would be a nice step to a complete  description of irregular sets of piecewise monotonic maps.

\subsection{Main results}
Let $T:X\to X$ be a Borel measurable map on a metric space $X$. 
 We say that $T$ is \emph{transitive} if there exists a point $x\in X$ whose forward orbit $\{ T^n (x) \mid n\geq 0\}$ is dense in $X$.
We denote by 
 $\mathcal M_T^{\mathrm{erg}}(X)$  
  the set of ergodic $T$-invariant Borel probability measures on $X$, endowed with the weak topology. 
A probability measure $\mu$
is called a periodic measure if there is a periodic point $p$ of period $n$ such that $\mu =  \big(\sum _{j=0}^{n-1} \delta _{T^j (p)} \big)/n$, where $\delta _y$ is the Dirac measure at $y\in X$, and we let 
$\mathcal M_T^{\mathrm{per}}(X)$ be  the set of periodic measures on $X$. 
Note that $\mathcal M_T^{\mathrm{per}}(X)\subset \mathcal M_T^{\mathrm{erg}}(X)$. 
Finally, a Borel measurable map   $T : [0,1] \to [0,1]$ on the interval $[0,1]$ is said to be a \emph{piecewise monotonic map} if there are disjoint open intervals $I_1, \ldots , I_k$ such that
$[0,1] \setminus \bigcup _{j=1}^kI_j $ is a finite set and 
 $T \vert _{I_j}$ is monotonic and continuous for each $1\leq j \leq k$.
We denote  by $h_{\mathrm{top}}(T,Z)$ the  topological entropy for a piecewise monotonic map $T$ on a
  Borel set $Z$, which was introduced by Hofbauer \cite{Hofbauer2010}  as a generalization of the Bowen's Hausdorff topological entropy for continuous maps  (see Subsection \ref{subsection:defofentropy} for precise definition).

Our main theorem is as follows:
\begin{maintheorem}\label{thm:main}
Let $T:[0,1]\to [0,1]$ be a transitive piecewise monotonic map.
Suppose that  $\mathcal M_T^{\mathrm{per}}([0,1])$ is dense in $\mathcal M^{\mathrm{erg}}_T([0,1])$. Then, for each   $\varphi \in \mathcal C([0,1])$, either $E(\varphi ) =\emptyset$  or $h_{\mathrm{top}}(T,E(\varphi )) = h_{\mathrm{top}}(T,[0,1])$.
\end{maintheorem}

Theorem \ref{thm:main} is a generalization of Thompson's theorem \cite[Theorem 5.1]{Thompson2012}
  for irregular sets of $\beta$-transformations. 
We will reduce Theorem \ref{thm:main} to  the following analogous result on coding spaces (see Subsection \ref{subsection:2.3} for definitions).

\begin{maintheorem}\label{thm:main2}
Let $T\colon [0,1]\to [0,1]$ be a transitive piecewise monotonic map,
$\Sigma_T^+$  the coding space of $T$ and $\sigma\colon\Sigma_T^+\to \Sigma_T^+$  the left shift operator.
Suppose that $\mathcal{M}_{\sigma}^{\mathrm{per}}(\Sigma_T^+)$ is dense in $\mathcal{M}_{\sigma}^{\mathrm{erg}}
(\Sigma_T^+)$. Then for each $\varphi\in \mathcal C(\Sigma_T^+)$, either $E(\varphi)=\emptyset$ or
$h_{\rm top}(\sigma,E(\varphi))=h_{\rm top}(\sigma,\Sigma_T^+)$.
\end{maintheorem}

\subsection{Applications}\label{subsection:application}
We will see in Remark \ref{rmk:1021} that  one can easily 
  construct \emph{non-transitive} piecewise monotonic maps $T$ and continuous maps $\varphi$ for which the dichotomy ``$E(\varphi )  = \emptyset$ or $h _{\mathrm{top}}(T,E(\varphi ) ) = h_{\mathrm{top}} (T,[0,1])$'' in Theorem \ref{thm:main} do not hold, 
   so the transitivity condition is indispensable for our purpose.
 On the other hand, 
 the  property 
 that $\mathcal M_T^{\mathrm{per}}([0,1])$ is dense in
$\mathcal M_T^{\mathrm{erg}}([0,1])$  has been intensively studied by many authors independently from irregular sets (e.g.~\cite{Hofbauer1987, Hofbauer1988, Blokh1995, HR1998})
and shown to hold for a large class of piecewise monotonic maps.  
We 
 recall that Hofbauer and Raith \cite{HR1998, Raith} proposed a problem asking whether  $\mathcal M_T^{\mathrm{per}}([0,1])$ is dense in
$\mathcal M_T^{\mathrm{erg}}([0,1])$ for any transitive piecewise monotonic maps with positive topological entropy.
The 
Hofbauer-Raith problem 
 has a positive answer  in the following three important cases.
 In these cases, 
 we can apply 
 Theorems \ref{thm:main} and \ref{thm:main2} if the map $T$ is transitive:

\begin{itemize}
\item The map $T$ is a \emph{continuous} map with positive topological entropy. It follows from \cite[Corollary 10.5]{Blokh1995} that $\mathcal{M}_T^{\mathrm{per}}([0,1])$ is dense in $\mathcal{M}_T^{\mathrm{erg}}([0,1])$ in this case. 
The result was motivated by the density of periodic measures for continuous maps with the specification property (\cite{DGS2006}).
\item  
The map $T$ is a \emph{monotonic mod one transformation} (that is, there exists a strictly increasing and continuous function $f : [0,1] \to \mathbb R$  such that $T(x) = f(x) \mod 1$; this is also called a \emph{Lorenz map}) with positive topological entropy. 
The density of periodic measures was proven in \cite[Theorem 2]{Hofbauer1987}. 
A special case of this class is a \emph{linear mod one transformation}, 
 which was first introduced by Parry (\cite{Parry1964}) and 
 defined by
\[
\label{d-mod-1}
T(x)=
\left\{
\begin{array}{ll}
\beta x+\alpha\ \; \text{mod}\ 1 & (x\in [0,1))\\
\displaystyle\lim_{y\rightarrow 1-0}(\beta y+\alpha\ \; \text{mod}\ 1) & (x=1)
\end{array}
\right.
\]
 with $\beta>1$ and $0\le \alpha <1$. 
When $\alpha =0$, the map $T$ is called the \emph{$\beta$-transformation} (its irregular sets were investigated in  \cite{Thompson2012}).
Another special case of monotonic mod one transformations appears in the Poincar\'e maps of geometric Lorentz flows
(refer to \cite{KLS2016} in which  irregular sets for geometric Lorentz flows  were shown to be residual).
\item The map $T$ has \emph{two} intervals of monotonicity  with positive topological entropy. 
That is, 
there is a point $0<a<1$ such that both 
 $T \vert _{(0,a)}$ and $T\vert _{(a,1)}$ are monotonic and continuous. 
 The density of periodic measures in this case was shown in \cite[Theorem 2]{HR1998}.
\end{itemize}
 In particular, we can and do apply Theorems \ref{thm:main} and \ref{thm:main2} to all transitive monotonic mod one transformations and all transitive piecewise monotonic maps with two monotonic pieces. To the best of our knowledge, this is the first result for Pesin-Pitskel type dichotomy for irregular sets of these transformations.

\section{Preliminaries}
\subsection{Topological entropy for non-compact sets}\label{subsection:defofentropy}
In this subsection, we recall the definition of the topological entropy for non-compact sets.
Let $X$ be a compact metric space with a metric $d$. We endow it with the Borel $\sigma$-field. 
Let $T\colon X\to X$  be a  measurable map.

Firstly, we  briefly recall the definition of the (Bowen's Hausdorff) topological entropy $h_{\mathrm{top}}(T,Z)$
for any continuous map $T$ and   Borel subset $Z\subset X$.
For $n\ge 1$, $x\in X$ and $\epsilon>0$, we let
$B_n(x,\epsilon)$ be the $\epsilon$-ball of center $x$ with respect to the $n$-th Bowen-Dinaburg metric $d_n$ 
given by $d_n(x,y) = \max \{ d(T^j(x),T^j(y)) \mid 0\leq j\leq n-1 \}$. 
For $s\in\mathbb{R}$, $L\in\mathbb{N}$ and $\epsilon>0$, we set
\[
M(Z,s,L,\epsilon)=\inf_{\Gamma}\left\{\sum_{B_{n_i}(x_i,\epsilon)\in\Gamma}e^{-sn_i}\right\}
\]
where the infimum is taken over all $\Gamma =\{B_{n_i}(x_i,\epsilon)\} _i$ being a finite or countable cover of $Z$ such that $x_i \in X$ and $n_i \geq L$ for all $i$,
with the convention $M(\emptyset ,s,L,\epsilon)=0$.
Since the quantity $M(Z,s,L,\epsilon)$ does not decrease with $L$, we can define $m(Z,s,\epsilon)$ given by
\[
m(Z,s,\epsilon)=\lim_{L \rightarrow \infty}
M(Z, s, L, \epsilon).
\]
Define 
\[
h_{\mathrm{top}}(T,Z,\epsilon)= \inf\{s \in \mathbb R \mid M(Z,s,\epsilon)=0\} =\sup\{s \in \mathbb R \mid M(Z,s,\epsilon)=\infty\},
\]
whose existence is easily seen 
 by the standard argument.
Finally we define
\[
h_{\mathrm{top}}(T,Z) =\lim_{\epsilon \rightarrow 0}h_{\mathrm{top}}(T,Z,\epsilon)
\]
and call it the \textit{topological entropy} of $Z$.

It is easy to see that $h_{\mathrm{top}}(T,Z_1)\le h_{\mathrm{top}}(T,Z_2)$ if $Z_1\subset Z_2\subset Y$.
Moreover, 
 $h_{\mathrm{top}}(T,X)$ coincides with the usual topological entropy
of $T: X\to X$  (we here mean by the usual topological entropy  the topological entropy in the sense of Adler-Konheim-McAndrew or Bowen-Dinaburg; see e.g.~\cite{HNP2008}). 
   In particular, by the classical variational principle 
    (for the Adler-Konheim-McAndrew  topological entropy in \cite{Goodman1971}; see also \cite[Corollary 8.6.1.(i)]{Walters2000}), 
   we have
\begin{equation}\label{eq:0202}
h_{\mathrm{top}}(T,X)=\sup_{\mu\in\mathcal{M}_T^{\mathrm{erg}}(X)}h(\mu),
\end{equation}
where $h(\mu)$ denotes the \textit{metric entropy}
of 
 $\mu$.

Next we recall the definition of the (Hofbauer's) topological entropy for a piecewise monotonic map $T$.
Let $\mathcal J$ be a partition of $[0,1]$ consisting of  monotonicity open intervals $I_1, \ldots , I_k$  of $T$ (i.e.~$I_1, \ldots , I_k$ are  disjoint,
$ [0,1] \setminus \bigcup _{j=1}^kI_j$ is a finite set and 
 $T \vert _{I_j}$ is monotonic and continuous for each $1\leq j \leq k$).
 Set $\mathcal J_\ell  =\bigvee _{i=0}^{\ell -1} T^{-i}\mathcal J$ for $\ell \geq 1$ and $\mathcal K_L =\bigcup _{\ell =L}^\infty \mathcal J_\ell $ for $L\geq 1$.
For $s \geq 0$ and $Z \subset [0, 1]$, we define
\[
\widetilde m(Z,s)=\lim_{L \rightarrow \infty}
\widetilde M(Z, s, L),
\]
where
\[
\widetilde M(Z, s, L) =\inf_{\Gamma}\sum_{B\in\Gamma}e^{-sn(B)},
\]
where  the infimum is taken over   the set of all finite or countable subsets $\Gamma$ of $\mathcal K_L$ with 
$Z \subset \bigcup _{B\in \Gamma} B$ and $n(B)$ is the maximal integer $\ell$ such that $B\in \mathcal J_\ell$.
It is shown in \cite[Lemma 2]{Hofbauer2010} that
for each subset $Z$ of $[0, 1]$, there is $s_0\geq  0$ independently of the choice of $\mathcal J$ such that $\widetilde m(Z,s) = \infty$ for all $s < s_0$ and
$\widetilde m(Z,s) =0$ for all $s > s_0$.
Furthermore, this value is proven to coincide with the Bowen's Hausdorff entropy of $Z$ whenever $Z$ is a subset of 
the maximal continuity set $X_T$, 
\begin{equation}\label{eq:1028b}
X_T=[0,1] \setminus \bigcup _{n=0}^\infty T^{-n} (\{ i_0, \ldots ,i_k\}),
\end{equation}
where $i_0, \ldots ,i_k$ are the endpoints of $I_1, \ldots ,I_k$
(\cite[Lemma 3]{Hofbauer2010}; note that $X_T$ is invariant and $T: X_T\to X_T$ is continuous).
Hence, we can consistently use the notation $h_{\mathrm{top}}(T,Z)$ to denote the value $s_0$.
Note also that
 $ \bigcup _{n=0}^\infty T^{-n} (\{ i_0, \ldots ,i_k\})$ is only countable, so 
 it follows from the argument above Lemma 1 of \cite{Hofbauer2010} that 
 \begin{equation}\label{eq:1028}
 h_{\mathrm{top}}(T,Z\cap X_T) = h_{\mathrm{top}}(T,Z)
 \end{equation}
 for any subset $Z$ of $[0,1]$.

\subsection{Symbolic dynamics}
Let $D$ be a countable set. 
Denote by $D^{\mathbb N}$ the one-sided  infinite product of $D$ 
 equipped with the product topology of the discrete topology of $D$.
Let $\sigma$ be the left shift operator of $D^{\mathbb N}$ 
 (i.e.~$(\sigma (x))_j = x_{j+1}$ for each $j\in \mathbb N$ 
 and $x= ( x_j)_{j\in \mathbb N} \in D^{\mathbb N}$). 
When a subset  $\Sigma ^+$ of $D^{\mathbb N}$ is $\sigma$-invariant and  closed,  
we call it 
 a \emph{subshift}, 
and $D$  the \emph{alphabet} of 
$\Sigma ^+$.
When 
 $\Sigma ^+$ is of the form
\[
\Sigma ^+ =\{ x\in D^{\mathbb N} \mid \text{$M_{x_j x_{j+1}} =1 $ for all $j\in \mathbb N$}\}
\]
with  a matrix $M= (M_{ij})_{(i,j)\in D^2}$  each entry of which is $0$ or $1$, we call $\Sigma ^+$ a \emph{Markov shift}.
When we emphasize the dependence of $\Sigma ^+$ on $M$, it is denoted by $\Sigma _M^+$, and $M$ is called the \emph{adjacency matrix} of $\Sigma^+_M$.

For a subshift $\Sigma^+$ on an alphabet $D$, let   $[x] = \left\{ y\in \Sigma^+ \mid (y_1, \ldots ,y_n) =x\right\}$ for each $x\in D^n$, $n\geq 1$, and set
$
\mathcal L(\Sigma^+ ) =\big\{ x\in \bigcup _{n\geq 1} D^n \mid [x] \neq \emptyset \big\} .
$
We say that a subshift $\Sigma^+$ on a finite alphabet satisfies the \emph{specification property} if
there is an integer $L>0$ such that for any $x ,y\in\mathcal{L}(\Sigma^+ )$,
one can find $z\in\mathcal{L}(\Sigma^+ )$ with $|z|\le L$ such that
$xzy\in\mathcal{L}(\Sigma^+ )$, where $|z |$ is the integer $n$ such that $z\in D^n$.
 (Notice that Bowen \cite{Bowen1974} originally called the specification property the above condition with $\vert z\vert =L$ instead of $\vert z\vert \leq L$, which seems to be more standard among dynamicists, 
but  the difference makes no influence on this paper and we permit us to adopt our definition for simplicity.)
The following is 
elementary but
 important in the proof of Theorem \ref{thm:main2}.

\begin{lem}\label{lem:key2}
Let $\Sigma^+_M$ be a transitive  Markov shift with an adjacency matrix $M$ on a finite alphabet $D$.
Then $\Sigma^+_M$ satisfies the  specification property.
\end{lem}
\begin{proof}
Since $\Sigma^+_M$ is transitive, for any $(i, j)\in D^2$, there exists
an integer $L(i ,j )>0$ such that $M^{L(i,j )}_{ij}>0$.
We set $L=\max_{(i, j)\in D^2}L(i, j)$.
Let $x,y\in\mathcal{L}(\Sigma^+_M)$ and denote by $n$ the length of $x$.
Since $M^{L(x_n , y_1 )}_{x_ny_1}>0$, we can find $z_1\cdots z_{\ell} \in\mathcal{L}(\Sigma^+_M)$ such that
$\ell =L(x_n, y_1) -1\le L$ and $M_{x_n z_1}=M_{z_{\ell } y_1}=1$, which imply that $xzy\in\mathcal{L}(\Sigma^+_M)$.
\end{proof}

\subsection{Markov diagram for piecewise monotonic maps}
\label{subsection:2.3}

Fix a transitive piecewise monotonic map $T: [0,1] \to [0,1]$ and let $X_T $ be the invariant set given in \eqref{eq:1028b}.
Define the coding map $\mathcal I : X_T \to \{ 1, \ldots ,k\} ^{\mathbb N}$ of $T$
 by
\[
(\mathcal I(x) )_j = \ell \quad \text{if $T^{j-1}(x) \in I_\ell$} .
\]
We note that $\mathcal{I}$ is well-defined and injective since $T$ is transitive.
Denote the closure of $\mathcal I(X_T) $ by $\Sigma ^+_T$ and call it the \emph{coding space} of $T$.
Observe that $\Sigma ^+_T$ is a compact $\sigma$-invariant set.
Again by the transitivity of $T$, we can easily see that $\Sigma_T^+$ is transitive. 
In what follows  we will construct Hofbauer's Markov diagram, which is a countable oriented graph
with subsets of $\Sigma_T^+$ as vertices.

Let $D\subset\Sigma_T^+$ be a closed subset with $D\subset [i]$ for some $1\le i\le k$.
We say that a non-empty closed subset $C\subset\Sigma_T^+$ is a \textit{successor} of $D$ if
$C=[j]\cap\sigma(D)$ for some $1\le j\le k$.
Now we define a set $\mathcal{D}$ of vertices by induction. First, we set
$\mathcal{D}_0=\{[1],\ldots,[k]\}$. If $\mathcal{D}_n$ is defined for $n\ge 0$, then we define $\mathcal{D}_{n+1}$ by
\[
\mathcal{D}_{n+1}=
\left \{C \subset\Sigma_T^+\mid \text{there exists }D\in\mathcal{D}_n\text{ such that }C\text{ is a
successor of }D \right\}.
\]
We note that $\mathcal{D}_n$ is a finite set for each $n$ since the number of successors of any
closed subset of $\Sigma_T^+$ is
at most $k$ by the definition. Finally, we set
\[
\mathcal{D} =\bigcup_{n\ge 0}\mathcal{D}_n.
\]
To get the oriented graph, which we call \emph{Hofbauer's Markov diagram}, we insert an arrow from every
$D\in\mathcal{D}$ to all of the  successors of $D$. We write $D\rightarrow C$ to denote that
$C$ is a successor of $D$.
We define a matrix 
$M(\mathcal{D})=(M_{D C})_{(D,C)\in\mathcal{D}^2}$ by
\[
M_{D C}=
\left\{
\begin{array}{ll}
1 & (D\rightarrow C), \\
0 & (\text{otherwise}).
\end{array}
\right.
\]
Then $\Sigma^+_{M(\mathcal{D})}=\{(D_i)_{i\in\mathbb{N}}\in \mathcal{D}^{\mathbb{N}} \mid D_i\rightarrow D_{i+1} \; \text{for all $i\in\mathbb{N}$} \}$
is a Markov shift with a countable alphabet $\mathcal{D}$ and an adjacency matrix $M(\mathcal{D})$.
We define $\Psi\colon \Sigma^+_{M(\mathcal{D})}\to \{1,\ldots,k\}^{\mathbb{N}}$ by 
\[
\Psi((D_i)_{i\in\mathbb{N}}) =(x_i)_{i\in\mathbb{N}}
\quad  \text{for $(D_i)_{i\in\mathbb{N}}\in\Sigma^+_{M(\mathcal{D})}$},
\]
where $1\le x_i\le k$ is the unique integer such that $D_i\subset [x_i]$ holds for each $i\in\mathbb{N}$.
Then it is clear that $\Psi$ is continuous, countable-to-one,
and satisfies $\Psi\circ\sigma=\sigma\circ\Psi$.
We remark that $\Sigma^+_{M(\mathcal{D})}$ is not transitive in general although $\Sigma_T^+$ is transitive.
We use the following two theorems shown by Hofbauer.

\begin{thm}$\mathrm{(}$\cite[Theorem 11]{Hofbauer1986}$\mathrm{)}$
\label{irreducible}
Suppose that $h_{\rm top}(\sigma,\Sigma_T^+)>0$. Then we can
find a subset $\mathcal{C}_0\subset\mathcal{D}$ such that
$\Sigma^+_{M(\mathcal{C}_0)}$ is transitive and $\Psi(\Sigma^+_{M(\mathcal{C}_0)})=\Sigma_T^+$.
Here $M(\mathcal{C})=(M_{DC}(\mathcal{C}))_{(D,C)\in\mathcal{C}^2}$ denotes the submatrix of $M(\mathcal{D})$ for $\mathcal{C}\subset\mathcal{D}$.
\end{thm}

\begin{thm}$\text{{\rm (}}$\cite[Theorems 1 and 2]{Hofbauer1981}$\text{{\rm )}}$
\label{maximal}
Suppose that $h_{\rm top}(\sigma,\Sigma_T^+)>0$ and
let $\mathcal{C}_0\subset \mathcal{D}$ be as in Theorem \ref{irreducible}. Then there
is a unique measure maximizing entropy both on $\Sigma^+_{M(\mathcal{C}_0)}$ and $\Sigma_T^+$,
i.e.~there is a unique probability measure $\widetilde{m}$ on $\Sigma^+_{M(\mathcal{C} _0)}$ and
a unique probability measure $m$ on $\Sigma_T^+$ such that
$h(\widetilde{m})=h(m)=h_{\rm top}(\sigma , \Sigma_T^+)$.
Furthermore, it holds that $m=\widetilde{m}\circ\Psi^{-1}$.
\end{thm}

To prove Theorem \ref{thm:main2}, we also use the following lemma:

\begin{lem}\label{lem:key}
Suppose that $h_{\rm top}(\sigma,\Sigma_T^+)>0$.
Let $\mathcal{C}_0$ be as in Theorem \ref{irreducible} and let $\mathcal F_1$, $\mathcal F_2$
be finite  subsets of $\mathcal{C}_0$. 
Then, one can find a subset $\mathcal F $ of $\mathcal{C}_0$ such that $\mathcal F_1\cup \mathcal F_2 \subset \mathcal F$ and  $\Psi ( \Sigma _{M(\mathcal F)}^+)$ satisfies the specification property.
\end{lem}

\begin{proof}
It is straightforward to see that, by the transitivity of $\Sigma^+ _{M(\mathcal{C}_0)}$,
there is a finite subset $\mathcal F $ of $\mathcal{C}_0$ such that $\mathcal F_1\cup \mathcal F_2 \subset \mathcal F$ and  $\Sigma^+ _{M(\mathcal F)}$ is a transitive Markov shift (on a finite alphabet $\mathcal F$).
Hence it follows from Lemma \ref{lem:key2} that
$\Sigma^+_{M(\mathcal F)}$ satisfies the specification property.
It is a well-known fact that the factor of a system with the specification property has the specification property (cf.~\cite[(21.4) Proposition (c)]{DGS2006}). 
By the property of $\Psi$ noted above, we have that
$\Psi\colon \Sigma^+_{M(\mathcal F)}\to\Psi(\Sigma^+_{M(\mathcal F)})$ is
a factor map, which completes the proof of the lemma.
\end{proof}

\section{The proofs of main theorems}

In this section, we  give the proofs of Theorems \ref{thm:main} and \ref{thm:main2}.
Let $\Sigma_T^+$ be as in Theorem \ref{thm:main2} and suppose that
$\mathcal{M}_{\sigma}^{\mathrm{per}}(\Sigma_T^+)$ is dense in $\mathcal{M}_{\sigma}^{\mathrm{erg}}(\Sigma_T^+)$.
Let $\varphi\in \mathcal C(\Sigma_T^+)$ and assume that $E(\varphi)\not=\emptyset$. 
We will show that $h_{\mathrm{top}}(\sigma,E(\varphi))=h_{\mathrm{top}}(\sigma,\Sigma_T^+)$. 
Since $h_{\rm top}(\sigma,\Sigma_T^+)=0$ immediately implies $h_{\rm top}(\sigma, E(\varphi ) )=0$  (and thus $h_{\mathrm{top}}(\sigma,E(\varphi))=h_{\mathrm{top}}(\sigma,\Sigma_T^+)$),
we may assume that $h_{\rm top}(\sigma,\Sigma_T^+)>0$. 
By  \cite[Lemma 1.6]{Thompson2010},
the assumption that $E(\varphi ) \neq \emptyset$ implies 
\begin{equation}
\label{e1}
\inf _{\mu \in \mathcal M_\sigma^{\rm erg} (\Sigma ^+ _T)} \int \varphi d\mu < \sup _{\mu \in \mathcal M_\sigma^{\rm erg} (\Sigma ^+ _T)} \int \varphi d\mu.
\end{equation}

It was proven in the previous works \cite{BS2000, CKS2005, Thompson2012} that if $\Sigma_T^+$ satisfies (a weaker form of) the specification property,
then $E(\varphi ) \neq \emptyset$ implies $h_{\rm top}(\sigma,E(\varphi))=h_{\rm top}(\sigma,\Sigma_T^+)$.
However, in our setting, 
 it 
 seldom happens  that $\Sigma _T^+$ has the specification
 property and it is unclear 
if there is a specification-like property satisfied by all  $\Sigma_T^+$,
which makes the proof of Theorem \ref{thm:main2} difficult.
To overcome the difficulty, we employed a strategy to find a subset of $\Sigma ^+_T$ which satisfies  the specification property with \emph{almost} full topological entropy:

\begin{prop}\label{prop:1}
For any $\epsilon >0$, there is a compact  $\sigma$-invariant set $\Sigma ^+_\epsilon \subset \Sigma ^+_T$ such that the following holds.
\begin{itemize}
\item[(I)] $\Sigma ^+_\epsilon$ satisfies the specification property.
\item[(II)] $\mathcal M_\sigma (\Sigma ^+_\epsilon )$ is nontrivial with respect to $\varphi$:
\[
\inf _{\mu \in \mathcal M_\sigma^{\rm erg} (\Sigma ^+_\epsilon )} \int \varphi d\mu < \sup _{\mu \in \mathcal M^{\rm erg}_\sigma (\Sigma ^+_\epsilon )} \int \varphi d\mu .
\] 
\item[(III)] $h_{\mathrm{top}} (\sigma, \Sigma ^+_\epsilon) \geq h_{\mathrm{top}} (\sigma,\Sigma_T^+) -\epsilon $.
\end{itemize}
\end{prop}
\begin{proof}
Let $m$ be the unique measure maximizing entropy on $\Sigma_T^+$ in Theorem \ref{maximal}. 
Then \eqref{e1} implies that there is an ergodic invariant probability measure $\mu$ on $\Sigma ^+_T$ such that
$
\int \varphi d\mu \neq \int \varphi dm.
$
Therefore, by the assumption that $ \mathcal M^{\mathrm{per}}_\sigma (\Sigma ^+_T)$ is dense in $ \mathcal M^{\mathrm{erg}}_\sigma (\Sigma ^+_T)$, 
one can find $\mu_{\mathrm{per}} \in \mathcal M^{\mathrm{per}}_\sigma (\Sigma ^+_T)$ such that
$
\int \varphi d\mu_{\mathrm{per}}\neq \int \varphi dm.
$
Let  $\alpha$ be a positive number such that $|\int\varphi d\mu_{\rm per}-\int\varphi d m|>2\alpha$ holds.
Choose an open neighborhood $\mathcal{U}$ of $m$ in $\mathcal{M}_{\sigma}(\Sigma_T^+)$ such that
$|\int\varphi d m - \int\varphi d\mu|\le \alpha$ whenever $\mu\in\mathcal{U}$.
Let $\mathcal{C}_0\subset\mathcal{D}$ be as in Theorem \ref{irreducible} and
$\widetilde{m}$  the unique measure maximizing entropy on $\Sigma^+_{M(\mathcal{C}_0)}$ such that  $m=\widetilde{m}\circ\Psi^{-1}$ in Theorem \ref{maximal}. 
Since $\Psi$ is continuous, we can find an open neighborhood $\mathcal V$ of $\widetilde{m}$
such that for any invariant measure $\widetilde{\mu}\in \mathcal V$, we have
$\widetilde{\mu}\circ\Psi^{-1}\in\mathcal{U}$.

Since $\Sigma^+_{M(\mathcal{C}_0)}$ is transitive, it follows from the entropy-approachability theorem for transitive Markov shifts \cite[Main Theorem]{Takahasi2019} that 
for any $\epsilon>0$, we can find a finite subset $\mathcal{F}_1\subset\mathcal{C}_0$ and
an ergodic measure $\widetilde{\nu}\in \mathcal{V}\cap \mathcal{M}^{\mathrm{erg}}_{\sigma}(\Sigma^+_{M(\mathcal{F}_1)})$ such that $h(\widetilde{\nu})\ge h(\widetilde{m})-\epsilon=h(m)-\epsilon$ holds. We set $\nu=\widetilde{\nu}\circ\Psi^{-1}$. Clearly, $\nu\in\mathcal{U}\cap
\mathcal{M}^{\rm erg}_{\sigma}(\Psi(\Sigma^+_{M(\mathcal{F}_1)}))$.
Since $\Psi\colon\Sigma^+_{M(\mathcal{F}_1)}\to\Psi(\Sigma^+_{M(\mathcal{F}_1)})$
is a countable-to-one factor map, we can see that $h(\widetilde{\nu})=h(\nu)$
(cf.~\cite[Theorem 2.1]{LW1977}).
Hence we have $h(\nu)\ge h_{\rm top}(\sigma,\Sigma_T^+)-\epsilon$.

Let $x\in\Sigma_T^+$ be a periodic point in the support of $\mu_{\mathrm{per}}$.
Then, 
by
$\Psi(\Sigma^+_{M(\mathcal{C}_0)})=\Sigma_T^+$ and \cite[Theorem 8]{Hofbauer1986}, 
there are finitely many vertices $C_1,\ldots,C_n\in\mathcal{C}_0$ such that
$\Psi(C_1\cdots C_nC_1\cdots C_n\cdots)=x$.
We set $\mathcal{F}_2=\{C_1,\ldots,C_n\}$. Since both $\mathcal{F}_1$ and $\mathcal{F}_2$ are finite subsets
of $\mathcal{C}_0$,
it follows from Lemma \ref{lem:key} that  one can find a subset $\mathcal F$ of $\mathcal{C}_0$ such that
$\mathcal{F}_1\cup\mathcal{F}_2\subset\mathcal{F}$ and
$\Psi(\Sigma^+_{M(\mathcal{F})})$ satisfies the specification property.

Set $\Sigma_{\epsilon}^+=\Psi(\Sigma^+_{M(\mathcal{F})})$.
We have already shown (I). Note that $\nu,\mu_{\rm per}\in\mathcal{M}^{\rm erg}_{\sigma}(\Sigma^+_{\epsilon})$.
Since $\nu\in\mathcal{U}$, we have
$|\int\varphi d\nu-\int\varphi d\mu_{\rm per}|\ge |\int\varphi dm-\int\varphi d\mu_{\rm per}|-
|\int\varphi d\nu-\int\varphi dm|\ge 2\alpha-\alpha>0$, which implies (II).
Finally by the variational principle in \eqref{eq:0202},
 \[
 h_{\mathrm{top}} (\sigma,\Sigma _\epsilon ^+) \geq h(\nu) 
   \ge h_{\mathrm{top}} (\sigma,\Sigma_T  ^+) -\epsilon,
 \]
 which implies (III) and completes the proof of Proposition \ref{prop:1}.
\end{proof}

\begin{remark}\label{rmk:3.2}
The first paragraph in the proof of Proposition \ref{prop:1} is the only place where  the density of periodic measures is used.
It is natural to ask if we can remove the condition by directly applying \cite[Main Theorem]{Takahasi2019} to a lift of the contrasting measure $\mu$ (i.e.~$\mu \in \mathcal M^{\mathrm{erg}}_\sigma (\Sigma ^+_T)$ satisfying $\int \varphi d\mu \neq \int \varphi dm$).
However, it is far from obvious that the measure $\mu$ can be lifted as a measure on $\Sigma ^+_{M(\mathcal C_0)}$ by $\Psi $ because the pushforward map $\widetilde\mu \mapsto \widetilde\mu \circ \Psi ^{-1}$ from $\mathcal{M}^{\mathrm{erg}}_{\sigma}(\Sigma^+_{M(\mathcal{C}_0)})$ to $\mathcal{M}^{\mathrm{erg}}_{\sigma}(\Sigma^+_T)$  is not surjective in general (although $\Psi : \Sigma ^+_{M(\mathcal{C}_0)} \to \Sigma^+_T$ is surjective; notice that the countable Markov shift $\Sigma^+_{M(\mathcal{C}_0)}$ is not compact  in general).
This is exactly the reason why we approximated $\mu$ by $\mu _{\mathrm{per}}$   in the proof of Proposition \ref{prop:1}.
\end{remark}

\begin{proof}[Proof of Theorem \ref{thm:main2}]
Fix $\epsilon >0$ and
 let $ \Sigma _\epsilon^+$
  be the compact $\sigma$-invariant subset of $\Sigma^+_T$ given in 
  Proposition \ref{prop:1}.
Then, it follows from \cite[Theorem 3.1]{CKS2005} that (I) and (II) of Proposition \ref{prop:1} imply the full topological entropy on $E(\varphi \vert _{ \Sigma _\epsilon^+})$:
\[
h_{\mathrm{top}}(\sigma,E(\varphi \vert _{ \Sigma _\epsilon^+}) )= h_{\mathrm{top}}(\sigma,\Sigma ^+_\epsilon) .
\]
So,  (III) leads to that
\[
h_{\mathrm{top}} (\sigma,\Sigma_T^+) \geq h_{\mathrm{top}}(\sigma,E (\varphi) ) \geq 
h_{\mathrm{top}}(\sigma,E(\varphi \vert _{ \Sigma _\epsilon^+}) ) \geq h_{\mathrm{top}} (\sigma,\Sigma_T^+) -\epsilon .
\]
Since $\epsilon$ is arbitrary, we get $  h_{\mathrm{top}}(\sigma,E (\varphi) ) = h_{\mathrm{top}} (\sigma,\Sigma_T^+)$, that is, the irregular set $E(\varphi) $ has  full topological entropy.
This completes the proof.
\end{proof}

\begin{proof}[Proof of Theorem \ref{thm:main}]
Let $(I_j)_{j=1}^k$ be
 the monotonicity intervals
 of $T$, and $i_0, \ldots ,i_k$  the endpoints of $I_1, \ldots ,I_k$ (see Subsection \ref{subsection:defofentropy}). 
Let $\varphi $ be a continuous function on $[0,1]$ such that $E(\varphi )\neq \emptyset$. 
Recall that $\mathcal M_T^{\mathrm{per}}([0,1])$ is supposed to be dense in $\mathcal M^{\mathrm{erg}}_T([0,1])$.
Thus, it follows from \cite[Theorem A]{Y} that
$\mathcal{M}_{\sigma}^{\mathrm{per}}(\Sigma_T^+)$ is also dense in $\mathcal{M}_{\sigma}^{\mathrm{erg}}(\Sigma_T^+)$.
Recall that $X_T$ is given in \eqref{eq:1028b}.
 
We first show that $E(\varphi)\cap X_T\not=\emptyset$.
Take a point $x\in E(\varphi)$. 
Arguing by contradiction, we assume that
$T^n(x)\not\in X_T$ for all $n\ge 0$. 
Then, $x\not\in X_T$, and thus  $T^{m} (x) = i_{j}$ with some $m\geq 0$ and $0\leq j\leq k$.
Since $T^{m} (x) \not\in X_T$ by assumption, we can find some $m'\geq 0$ and $0\leq j'\leq k$ such that $T^{m'+m}(x) =i_{j'}$. By repeating this argument (at most) $k$-times, we conclude that each $i_j$ is a periodic point.
It
 contradicts with that $x\in E(\varphi)$ and $T^{m} (x) = i_{j}$.
Hence there is an integer $n\ge 0$ such that $T^n(x)\in X_T$.
Since $E(\varphi )$ is an invariant set, this implies that
$E(\varphi)\cap X_T\neq \emptyset$.

Note that the coding map
$\mathcal{I}\colon X_T\to \Sigma_T^+$ is injective by the transitivity of $T$ (refer to \S6 in \cite{Y}).
Hence one can see that the set
\begin{equation}
\label{rev}
\bigcap_{n\ge 0}{\rm cl}(T^{-n}(I_{x _{n+1}}))
\end{equation}
is a unit set for any $(x_n)_{n\in\mathbb{N}}\in\Sigma_T^+$.
Here ${\rm cl}(A)$   denotes the closure  of the set $A$.
Then we define a map $\Phi\colon\Sigma_T^+\to [0,1]$
by $\Phi((x_n)_{n\in\mathbb{N}})=y$, where $y$ is the unique element of the unit set in (\ref{rev}).
Then it is well known that $\Phi$ is continuous, surjective and $\Phi\circ\sigma(x)=T\circ\Phi(x)$
for any $x\in\mathcal{I}(X_T)$. Moreover, we have $\Phi(\mathcal{I}(X_T))=X_T$ and
the restricted map $\Phi\colon\mathcal{I}(X_T)\to X_T$ is the inverse map of $\mathcal{I}\colon X_T\to \mathcal{I}(X_T)$.
Hence we have 
$
\varphi\circ\Phi\in \mathcal{C}(\Sigma_T^+)$ and 
\[
E(\varphi\circ\Phi)\cap \mathcal{I}(X_T)=\mathcal I(E(\varphi)\cap X_T).
\]
This  implies that
$E(\varphi\circ\Phi)\not=\emptyset$ since $E(\varphi)\cap X_T \neq \emptyset$, and that $h_{\rm top}(T,E(\varphi)\cap X_T)=h_{\rm top}(\sigma,E(\varphi\circ\Phi))$ since $\Sigma_T^+\setminus \mathcal{I}(X_T)$ is countable. 
Hence, it follows from  Theorem \ref{thm:main2} that
\[
h_{\rm top}(\sigma,E(\varphi\circ\Phi)) =h_{\rm top}(\sigma,\Sigma_T^+).
\]
Therefore, 
  by virtue of \eqref{eq:1028}, we have 
\begin{align*}
h_{\rm top}(T,E(\varphi))  &= h_{\rm top}(T,E(\varphi)\cap X_T)\\
&=h_{\rm top}(\sigma,\Sigma_T^+)
=h_{\rm top}(T,X_T)=h_{\rm top}(T,[0,1]), 
\end{align*}
which completes the proof.
\end{proof}

\begin{remark}\label{rmk:1021}
Let $T: [0,1]\to [0,1]$ be a (not necessarily transitive) piecewise monotone map  for which  $\mathcal M_T^{\mathrm{per}} ([0,1])$ is dense in $\mathcal M_T^{\mathrm{erg}} ([0,1])$. 
Assume that there are 
mutually disjoint invariant intervals $J_1, \ldots , J_N$ with $N\in \mathbb N$
 such that $ \bigcup 
 _{j=1}^N J_j =[0,1]$
and the restriction of $T$ on $J_j$ is transitive  for each $1\leq j\leq N$.
Then for any continuous function $\varphi : [0,1]\to \mathbb R$
  for which $E(\varphi )\neq \emptyset$,
it follows from 
Theorem \ref{thm:main} and \eqref{eq:1028}
 together with
  \cite[Lemma 1.6]{Thompson2010} that 
 \begin{multline*}
 h_{\mathrm{top}}(T, E(\varphi) ) =\sup \Big\{ h_{\mathrm{top}}(T, J_j ) \mid
 1\leq j\leq N \; \text{such that} \\ 
  \inf _{\mu \in \mathcal M_T^{\mathrm{erg}}(J_j\cap X_T)} \int \varphi d\mu <\sup _{\mu \in \mathcal M_T^{\mathrm{erg}}(J_j\cap X_T)} \int \varphi d\mu  \Big\} .
 \end{multline*}
As a consequence, 
 one can easily 
  construct \emph{non-transitive} piecewise monotonic maps $T$ and continuous functions  $\varphi$ for which the dichotomy ``$E(\varphi ) = \emptyset$ or $h _{\mathrm{top}}(T,E(\varphi ) ) = h_{\mathrm{top}} (T,[0,1])$'' do not hold (for example, take $J_1$, $J_2$ with   $h_{\mathrm{top}}(T, J_1) < h_{\mathrm{top}}(T,J_2) $ and $\varphi $   supported  on $J_1$ with $\inf _{\mu \in \mathcal M_T^{\mathrm{erg}}(J_1\cap X_T)} \int \varphi d\mu <\sup _{\mu \in \mathcal M_T^{\mathrm{erg}}(J_1\cap X_T)} \int \varphi d\mu  $). 
  \end{remark}

\section*{Acknowledgments}
We would like to express our deep gratitude to Paulo Varandas for many valuable comments. 
 We also would like to thank
Yong Moo Chung for many suggestions.
This work was partially supported by JSPS KAKENHI
Grant Numbers 19K14575, 19K21834 and 18K03359.

\bibliography{NY}

\end{document}